\documentclass[11pt,a4paper]{amsart}

\usepackage[latin1]{inputenc}
\usepackage[english]{babel}
\usepackage{amsmath}
\usepackage{bm}
\usepackage{amssymb}
\usepackage{graphicx}
\usepackage{braket}
\usepackage[plainpages=false,colorlinks,hyperindex,bookmarksopen,linkcolor=red,citecolor=blue,urlcolor=blue]{hyperref}
\usepackage{epstopdf}
\usepackage{braket}
\usepackage[hmargin=3cm,vmargin={3.5cm,4cm}]{geometry}

\theoremstyle{theorem}
\newtheorem{theorem}{Theorem}[section]
\newtheorem{lemma}{Lemma}[section]

\newtheorem{corollary}{Corollary}[section]

\theoremstyle{definition}                           %stile roman
\newtheorem{definition}{Definition}[section]                   %definizione ambiente definizione

\theoremstyle{remark}                             %stile per osservazioni
\newtheorem{remark}{Remark}[section]              %definizione ambiente osservazione

\usepackage{color}

\usepackage{mathtools,slashed}

\newcommand{\be}{\begin{eqnarray}}
\newcommand{\ee}{\end{eqnarray}}

\def\eu{{\ensuremath{\mathrm{e}}}}
\def\iu{{\ensuremath{\mathrm{i}}}}
\def\du{\,{\ensuremath{\mathrm{d}}}}

\def\RR{\vbox {\hbox to 8.9pt {I\hskip-2.1pt R\hfil}}\;}

\allowdisplaybreaks
\begin{document}

\title[On Le Roy type function \dots]{On a generalized three-parameter Wright \\ function of the Le Roy type}
	
		\author[R. Garrappa]{Roberto Garrappa${}^{1}$}
		\address{${}^{1}$ 
		Department of Mathematics,
 		University of Bari, 
		Via E. Orabona 4, 70126 Bari, ITALY
		and 
		the INdAM Research group GNCS}	
 		\email{roberto.garrappa@uniba.it}			
		
			\author[S. Rogosin]{Sergei Rogosin${}^{2}$}
		\address{${}^{2}$ 
		Department of Economics, Belarusian State University, 
        Nezavisimosti Ave 4, BY-220030 Minsk, BELARUS}	
 		\email{rogosin@bsu.by}

		\author[F. Mainardi]{Francesco Mainardi${}^{3}$}
		\address{${}^{3}$ 
		Department of Physics $\&$ Astronomy,
 		University of Bologna,
		and
		INFN, Via Irnerio 46, \\ I-40126 Bologna, ITALY}	
 		\email{francesco.mainardi@bo.infn.it}				
	
\begin{abstract}
Recently S. Gerhold and R. Garra-F. Polito independently introduced a new function related to the special functions of Mittag-Leffler family.
This function is a generalization of the function studied by \'E. Le Roy
in the period 1895-1905 in connection with the problem of analytic
continuation of power series with a finite radius of convergence.
In our note we obtain two integral representations of this special function, 
calculate its Laplace transform, determine an asymptotic expansion of this function on the negative semi-axis (in the case of an integer third $\gamma$) and provide its continuation to the case of a negative first parameter $\alpha$. An asymptotic result is illustrated by numerical calculations.  Discussion on possible further studies and open questions are also presented.
\end{abstract}

\thanks{\textbf{In}: \emph{Fract. Calc. Appl. Anal.}, Vol. \textbf{20}, No 5 (2017), pp. 1196-1215, 
DOI: 10.1515/fca-2017-0063 at \href{https://doi.org/10.1515/fca-2017-0063}{https://www.degruyter.com/view/journals/fca/20/5/article-p1196.xml}}

    \maketitle

 \section{Introduction}\label{Intro}
\setcounter{section}{1}
\setcounter{equation}{0}\setcounter{theorem}{0}

In two recent papers S. Gerhold \cite{Ger12} and, independently, R. Garra and F. Polito \cite{GarPol13} introduced a new special function
\begin{equation}
\label{GGP1}
F_{\alpha,\beta}^{(\gamma)}(z) = \sum\limits_{k=0}^{\infty} \frac{z^k}{[\Gamma(\alpha k + \beta)]^{\gamma}},\quad z\in {\mathbb C}, \quad \alpha, \beta, \gamma \in {\mathbb C},
\end{equation}
which turns out to be an entire function of the complex variable $z$ for all values of the parameters such that ${\mathrm{Re}}\, \alpha > 0$, $\beta\in {\mathbb R}$ and $\gamma > 0$.

This function $F_{\alpha,\beta}^{(\gamma)}(z)$ is closely related to the family of the 2-parameter Mittag-Leffler function (see the recent monograph \cite{GKMR})
\begin{equation}\label{def_ML}
E_{\alpha,\beta} (z) = \sum_{k=0}^\infty \frac{z^k}{\Gamma(\alpha k + \beta)},
 \quad  \alpha\in {\mathbb C},\quad {\mathrm{Re}}\, \alpha > 0, \quad \beta\in {\mathbb R},
\end{equation}
and its multi-index extensions (with $2m$-parameters, $m =1, 2, \ldots$).
The function (\ref{def_ML}) is named after the great  Swedish mathematician {G\"{o}sta  Magnus Mittag-Leffler} (1846-1927)  who defined it in 1-parameter case ($E_{\alpha} (z)$ with $\beta = 1$) by a power series
and studied its properties in 1902-1905 (see detailed description in  \cite{GKMR}). As a matter of fact, Mittag-Leffler introduced the function $E_{\alpha} (z)$ for the purposes of his method for summation od divergent series. Later, the function (\ref{def_ML}) was recognized as the ``Queen function of fractional calculus'' \cite{GorMai97,MaiGor07,Rog15} for its basic role for analytic solutions of fractional order integral and differential equations and systems.

In recent decades successful applications of the Mittag-Leffler function and its
generalizations in problems of physics,
biology, chemistry, engineering and other applied sciences made it better known among scientists. A considerable literature is devoted to the investigation of the analyticity properties of this function; among the references of \cite{GKMR} there are
quoted several authors who, after Mittag-Leffler, have investigated
such a function from a pure mathematical, application oriented and numerical point of view.

The function (\ref{GGP1}) is also related to the so-called Le Roy function
\begin{equation}
\label{LeR}
R_{\rho}(z) = \sum\limits_{k=0}^{\infty} \frac{z^k}{[(k + 1)!]^{\gamma}},\;\;\; z\in {\mathbb C},
\end{equation}
which was used in \cite{LeRoy} to study the asymptotics of the analytic continuation of the sum of power series. This reason for the origin of (\ref{LeR}) sounds somehow close to the Mittag-Leffler idea to introduce the function $E_{\alpha} (z)$ for the aims of analytic continuation (we have to note that Mittag-Leffler and Le Roy were working on this idea in competition).
The Le Roy function is involved in the solution of problems of various types; in particular it has been recently used in the construction of a Convey-Maxwell-Poisson distribution which is important due to its ability to model count data with different degrees of over- and under-dispersion \cite{SLSPL}.

For shortness we use in this paper the  name Le Roy type-function for $F_{\alpha,\beta}^{(\gamma)}(z)$ defined by (\ref{GGP1}).  We mention generalization of the Wright function in the title of our paper since $F_{\alpha,\beta}^{(\gamma)}(z)$ has some properties
 similar to those of the Wright function (or better to say, of the Fox-Wright function $_{p}\Psi_{q}(z)$, see below).

When $\gamma$ in (\ref{GGP1})  is a positive integer
($\gamma = m\in {\mathbb N}$), the Le Roy type-function $F_{\alpha,\beta}^{(m)}(z)$ becomes a special case of the $2m$-parameter Mittag-Leffler function (studied by Luchko and Kiryakova in \cite{Al-BLuc95,Kir99,Kir10b}, and also for more general values of $\alpha$ by Kilbas et al., see \cite{KiKoRo13})
\begin{equation}
\label{LKK}
	F_{\alpha,\beta}^{(m)}(z)
	= E{((\alpha,\, \beta)}_m;z)
	= \sum\limits_{k=0}^{\infty}\frac{z^k}{\prod\limits_{j=1}^{m}\Gamma(\alpha_j k+\beta_j)}, \quad \alpha_j = \alpha, \, \beta_j = \beta.
\end{equation}

The study of the asymptotic behavior of the Le Roy type-function is of special interest due to existing and perspective applications. Thus, the work in \cite{Ger12} was devoted to study the asymptotic properties of $F_{\alpha,\beta}^{(\gamma)}(z)$ as an entire function in some sectors of the complex plane (it was implicitly shown that this function has order $\rho = 1/{\alpha \gamma}$ and type $\sigma = \gamma$). The main result reads that
 (\ref{GGP1}) has the following asymptotic
\begin{equation}
\label{Ger2}
F_{\alpha,\beta}^{(\gamma)}(z) \sim \frac{1}{\alpha \sqrt{\gamma}} (2 \pi)^{(1 - \gamma)/2} z^{(\gamma - 2 \beta \gamma + 1)/{2 \alpha \gamma}} \eu^{\gamma z^{1/{\alpha \gamma}}}, \;\;\; |z| \rightarrow \infty,
\end{equation}
in the sector
\begin{equation}
\label{Ger1}
|\arg\, z| \leq \left\{
\begin{array}{ll}
\frac{1}{2} \alpha \gamma \pi - \varepsilon, & 0 < \alpha \gamma < 2, \\
& \\
(2 - \frac{1}{2} \alpha \gamma) \pi - \varepsilon, & 2 \leq \alpha \gamma < 4, \\
& \\
0, & 4 \leq \alpha \gamma ,
\end{array}
\right.
\end{equation}
with $\varepsilon$ being an arbitrary small number. This result was obtained by using the saddle point method as described in \cite{Evg79} and the purpose of the analysis in  \cite{Ger12} was to apply asymptotics in order to deliver certain holonomicity results for power series.

In \cite{GarPol13} the function $F_{\alpha,\beta}^{(\gamma)}(z)$ is considered from the operational point of view. More specifically, the properties of this function are studied in relation to some integro-differential operators involving the Hadamard fractional derivatives (e.g., see \cite[Sec. 18.3]{SaKiMa93}) or hyper-Bessel-type operators. By using these properties, the operational (or formal) solutions to certain boundary and initial value problems for fractional differential equations are derived. An application of the developed technique to a modified Lamb-Bateman integral equation is also presented.

The aim of the present paper is to provide a further study of the Le Roy type-function. In particular, in Section \ref{S:MB} we obtain two types of integral representations of this function, in Section \ref{S:Laplace} we derive the Laplace transform of $F_{\alpha,\beta}^{(\gamma)}$  and we find in Section \ref{S:Asympt} an asymptotic expansion on the negative semi-axis (for integer parameter $\gamma$) which is illustrated by means of some plots. An extension of $F_{\alpha,\beta}^{(\gamma)}(z)$ to negative values (negative real part) of the first parameter $\alpha$ is obtained in Section \ref{S:Negative} and we conclude the paper by discussing possible further studies and posing open questions in Section \ref{S:Disc}.

\section{Integral representations of the Le Roy type-function}\label{S:MB}
\setcounter{section}{2}
\setcounter{equation}{0}\setcounter{theorem}{0}

The Le Roy type-function $F_{\alpha,\beta}^{(\gamma)}(z)$ is an entire function of the complex variable $z$ in the case of positive values of all three parameters
since in this case

$$
\limsup\limits_{n \rightarrow \infty} |c_n|^{\frac{1}{n}} = 0
, \quad
c_{n} = \frac{1}{[\Gamma(\alpha n + \beta)]^{\gamma}} .
$$

It is however not hard to see that the above property holds also under more general assumptions on the parameters: %, namely
  ${\mathrm{Re}}\, \alpha > 0$, $\beta \in {\mathbb C}$, $\gamma > 0$.

 The order $\rho$ and the type $\sigma$ of the Le Roy type-function can be found directly from the series representation  (\ref{GGP1}) by using standard formulas for $\rho := \rho_{F}$ and  $\sigma  := \sigma_{F}$ valid for any entire function of the form (see e.g. \cite[p. 287]{GKMR})

$$
F(z) = \sum\limits_{n = 0}^{\infty} c_n z^n,
$$
namely
\begin{equation}
\label{order}
\rho  = \limsup\limits_{n \rightarrow \infty} \frac{n \log\, n}{\log \frac{1}{|c_n|}},
\end{equation}
\begin{equation}
\label{type}
(\sigma e \rho)^{\frac{1}{\rho}} = \limsup\limits_{n \rightarrow \infty} \left(n^{\frac{1}{\rho}} |c_n|^{\frac{1}{n}}\right).
\end{equation}

By using the Stirling formula for the Gamma function (e.g. \cite[p. 254]{GKMR})
\begin{equation}
\label{Stirling}
\Gamma(\alpha z + \beta) \approx \sqrt{2 \pi} e ^{- \alpha z} (\alpha z)^{\alpha z + \beta - 1/2}\left(1 + O\left(\frac{1}{z}\right)\right)
\end{equation}
we get the following result which helps us to predict the maximal possible growth of the function $F_{\alpha,\beta}^{(\gamma)}(z)$.

\begin{lemma}
\label{order-type}
Let $\alpha, \beta, \gamma > 0$. The order and type of the entire Le Roy type-function $F_{\alpha,\beta}^{(\gamma)}(z)$ are
\begin{equation}\label{OrderTypeGGP}
\rho_{F_{\alpha,\beta}^{(\gamma)}} = \frac{1}{\alpha \gamma},
\quad
\sigma_{F_{\alpha,\beta}^{(\gamma)}} = \alpha.
\end{equation}

These formulas still holds for any $\alpha, \beta, \gamma$ such that ${\mathrm{Re}}\, \alpha > 0$, $\beta \in {\mathbb C}$, $\gamma > 0$ if the parameter $\alpha$ is replaced with ${\mathrm{Re}}\, \alpha$  in (\ref{OrderTypeGGP}).
\end{lemma}

One of the important tools to study the behavior of Mittag-Leffler type functions is their Mellin-Barnes
integral representation (see e.g. \cite{GKMR}, \cite{ParKam01}). Below we establish two integral representations for our function $E_{\alpha,\beta}^{(\gamma)}$
which use the technique similar to that in the Mellin-Barnes formulas. Anyway, we have to note that our integral representations cannot be always called Mellin-Barnes type representations since in the case of noninteger $\gamma$ the integrands in these formulas contain a multi-valued in $s$ function $[\Gamma(\alpha s + \beta)]^{\gamma}$.

For simplicity we consider here and in what follows the function $F_{\alpha,\beta}^{(\gamma)}(z)$ with positive values of all parameters ($\alpha, \beta, \gamma > 0$). In this case the function $\Gamma(\alpha s + \beta)$ is a meromorphic function of the complex variable $s$ with just simple poles at points $s = - \frac{\beta + k}{\alpha}$, $k = 0, 1, 2, \ldots$. We fix the principal branch of the multi-valued function $[\Gamma(\alpha s + \beta)]^{\gamma}$ by drawing the cut along the negative semi-axes starting from $-\frac{\beta}{\alpha}$, ending at $-\infty$ and by supposing that $[\Gamma(\alpha x + \beta)]^{\gamma}$ is positive for all positive $x$. Let also
the function $(-z)^s$ be defined in the complex plane cut along negative semi-axis and

$$
(-z)^s = \exp \{s [\log |z| + \iu \arg\, (-z)]\},
$$
where $\arg\, (-z)$ is any arbitrary chosen branch of ${\mathrm{Arg}}\, (-z)$.

\begin{theorem}
\label{Mellin-Barnes}
Let $\alpha, \beta, \gamma > 0$ and $[\Gamma(\alpha s + \beta)]^{\gamma}$, $(-z)^s$ be the  described branches of the corresponding multi-valued functions. Then the Le Roy type-function possesses the following ${\mathcal L}_{+\infty}$-integral representation 
\begin{equation}
\label{Mellin-Barnes1}
F_{\alpha,\beta}^{(\gamma)}(z) = \frac{1}{2\pi \iu} \int\limits_{{\mathcal L}_{+\infty}} \frac{\Gamma(- s) \Gamma(1 + s)}{[\Gamma(\alpha s + \beta)]^{\gamma}} (- z)^{s} \du s + \frac{1}{[\Gamma(\beta)]^{\gamma}},\;\;\; z\in {\mathbb C}\setminus (-\infty, 0],
\end{equation}
where ${\mathcal L}_{+\infty}$ is a right loop situated in a horizontal strip starting at the point $+ \infty + \iu \varphi_1$ and terminating at the point $+ \infty + \iu \varphi_2$, $-\infty < \varphi_1 < 0< \varphi_2 < + \infty$, crossing the real line at a point $c, 0 < c < 1$.
\end{theorem}
\begin{proof}
The chosen contour ${\mathcal L}_{+\infty}$ separate the poles $s = 1, 2, \ldots$ of the function $\Gamma(- s)$ and $s = -1, -2, \ldots$ of the function $\Gamma(1 + s)$, together with the pole at $s =0$ of the function $\Gamma(- s)$. So, the integral locally exists (see, e.g., \cite[p. 1]{KilSai}, \cite[p. 66]{ParKam01}).

Now we prove the convergence of the integral in (\ref{Mellin-Barnes1}). To this purpose we use the reflection formula for the Gamma function \cite[p. 250]{GKMR}
\begin{equation}
\label{Mellin-Barnes2}
\Gamma(z) \Gamma(1 - z) = \frac{\pi}{\sin\, \pi z}, \;\;\; z\not\in {\mathbb Z},
\end{equation}
and the %classical
 Stirling formula (\ref{Stirling}) which holds for any $\alpha, \beta > 0$.
%\begin{equation}
%\label{Mellin-Barnes3}
%\Gamma(az + b) = \sqrt{2 \pi} \eu^{- a z} (a z)^{a z + b - 1/2} \left(1 +O\left(\frac{1}{z}\right)\right),\quad |\arg\, z| < \pi.
%\end{equation}

First we note that on each ray $s = x + \iu \varphi_j$, $j = 1, 2$, $\varphi_j > 0$, it is
\[
	\begin{aligned}
		\Gamma(- s) \Gamma(1 + s)
		&= -\frac{1}{s} \Gamma(1 - s) s \Gamma(s)
		 = \frac{- \pi}{\sin\, \pi s} \\
		&= \frac{- 2 \pi \iu}{\cos\, \pi x \left(\eu^{- \pi \varphi_j} - \eu^{\pi \varphi_j}\right)
		   + \iu \sin\, \pi x \left(\eu^{- \pi \varphi_j} + \eu^{\pi \varphi_j}\right)}  \
	\end{aligned}
\]
and hence

$$
|\Gamma(- s) \Gamma(1 + s)| = \frac{\pi}{\sqrt{\sinh^2 \pi \varphi_j + \sin^2 \pi x}}.
$$

Since
$$
\sinh^2 \pi \varphi_j + \sin^2 \pi x > \sinh^2 \pi \varphi_j > 0,
$$
it gives
\begin{equation}
\label{Mellin-Barnes4}
|\Gamma(- s) \Gamma(1 + s)|  \leq C_1,\;\;\; s\in {\mathcal L}_{+\infty}.
\end{equation}

Next, it follows from (\ref{Stirling}) that
\[
        \begin{aligned}
                \log\, [&\Gamma(\alpha s + \beta)]^{\gamma} = \\
                &= \gamma \Bigl[\frac{1}{2}\log\, 2\pi + (\alpha s + \beta - 1/2) \log\, \alpha s
                  - \alpha s + \log\, \left(1 +\mathcal{O}\left(z^{-1}\right)\right)\Bigr] = \\
                &= \gamma \frac{1}{2}\log\, 2\pi  + \gamma (\alpha x + \iu \alpha \varphi_j + \beta - 1/2) \log\, (\alpha x + \iu \alpha \varphi_j) \\
                &\phantom{= } \, - \gamma \alpha (x + \iu \varphi_j) + \gamma  \log\, \left(1 +\mathcal{O}\left(z^{-1}\right)\right). \\
        \end{aligned}
\]

Hence
\[
        \begin{aligned}
                \log\, |[\Gamma(\alpha s &+ \beta)]^{\gamma}|
                        = {\mathrm{Re}} \log\, [\Gamma(\alpha s + \beta)]^{\gamma} =\\
                        &= \gamma \frac{1}{2}\log\, 2\pi - \gamma \alpha x + \gamma  (\alpha x + \beta - 1/2) (\log\, \alpha + \log\, | x + \iu \varphi_j|) \\
                        &\phantom{= } - \gamma \alpha \varphi_j \arg\, (x + \iu \varphi_j) + \gamma {\mathrm{Re}} \log\, \left(1 +{\mathcal O}\left(z^{-1}\right)\right) \
                \end{aligned}
\]
and therefore
\begin{equation}
\label{Mellin-Barnes5}
|[\Gamma(\alpha s + \beta)]^{\gamma}| = C_2 \eu^{- \gamma \alpha x} \alpha^{\gamma  x} |x + \iu \varphi_j|^{\gamma  (\alpha x + \beta - 1/2)}.
\end{equation}

At last
\begin{equation}
\label{Mellin-Barnes6}
|(-z)^{s}| = |z|^{x} \eu^{- \varphi_j \arg\, (-z)},\; z = x + \iu \varphi_j.
\end{equation}

The obtained asymptotic relations (\ref{Mellin-Barnes4})--(\ref{Mellin-Barnes6}) give us the convergence of the integral in (\ref{Mellin-Barnes1}) for each fixed $z\in {\mathbb C}\setminus (-\infty, 0]$.

Finally, we evaluate the integral by using the residue theorem (since  the poles $s = 1, 2, \ldots$ remain right to bypass of the contour ${\mathcal L}_{+\infty}$):
$$
%\begin{equation}
%\label{Mellin-Barnes7}
\frac{1}{2\pi \iu} \int\limits_{{\mathcal L}_{+\infty}} \frac{\Gamma(- s) \Gamma(1 + s)}{[\Gamma(\alpha s + \beta)]^{\gamma}} (- z)^{s} \du s = - \sum\limits_{k=1}^{\infty} {\mathrm{Res}}_{s = k} \left[\frac{\Gamma(- s) \Gamma(1 + s)}{[\Gamma(\alpha s + \beta)]^{\gamma}} (- z)^{s}\right].
$$
%\end{equation}
Since
$$
{\mathrm{Res}}_{s = k} \Gamma(- s) = - \frac{(-1)^k}{k!},\;\;\; \Gamma(1 + k) = k!,
$$
then we obtain the final relation
$$
%\begin{equation}
%\label{Mellin-Barnes7b}
\frac{1}{2\pi \iu} \int\limits_{{\mathcal L}_{+\infty}} \frac{\Gamma(- s) \Gamma(1 + s) (- z)^{s} \du s}{[\Gamma(\alpha s + \beta)]^{\gamma}} = \sum\limits_{k=1}^{\infty} \frac{z^k}{[\Gamma(\alpha k + \beta)]^{\gamma}} = F_{\alpha,\beta}^{(\gamma)}(z) - \frac{1}{[\Gamma(\beta)]^{\gamma}}.
$$
%\end{equation}
\end{proof}

Now we get another form of the representation of the Le Roy type-function via generalization of the Mellin-Barnes integral. We consider the multi-valued function $[\Gamma(\alpha (-s) + \beta)]^{\gamma}$ and fix its principal branch by drawing the cut along the positive semi-axis starting from $\frac{\beta}{\alpha}$ and ending at $+\infty$ and supposing that $[\Gamma(\alpha (-x) + \beta)]^{\gamma}$ is positive for all negative $x$. We also define the function $z^{-s}$ in the complex plane cut along positive semi-axis and $
z^{-s} = \exp \{(- s) [\log |z| + \iu \arg\, z]\},
$
where $\arg\, z$ is any arbitrary chosen branch of ${\mathrm{Arg}}\, z$.

%{The corresponding Mellin-Barnes type representation has the following form.}
\begin{theorem}
\label{Mellin-BarnesT2}
Let $\alpha, \beta, \gamma > 0$ and $[\Gamma(\alpha (-s) + \beta)]^{\gamma}$, $z^{-s}$ be the  described branches of the corresponding multi-valued functions. Then the Le Roy type-function possesses the following ${\mathcal L}_{-\infty}$-integral
representation 
\begin{equation}
\label{Mellin-BarnesT2_1}
F_{\alpha,\beta}^{(\gamma)}(z) = \frac{1}{2\pi \iu} \int\limits_{{\mathcal L}_{-\infty}} \frac{\Gamma(s) \Gamma(1 - s)}{[\Gamma(\alpha (-s) + \beta)]^{\gamma}} z^{-s} \du s + \frac{1}{[\Gamma(\beta)]^{\gamma}},
\end{equation}
where ${\mathcal L}_{-\infty}$ is a left loop situated in a horizontal strip starting at the point $- \infty + \iu \varphi_1$ and terminating at the point $- \infty + \iu \varphi_2$, $-\infty < \varphi_1 < 0< \varphi_2 < + \infty$, crossing the real line at a point $c, -1 < c < 0$.
\end{theorem}

The proof repeats all the arguments of the proof to Theorem \ref{Mellin-Barnes} by using the
behavior of the integrand on the contour ${\mathcal L}_{-\infty}$ and calculating the residue at the poles $s = - 1, - 2, \ldots$.%, which are encircled by ${\mathcal L}_{-\infty}$ remaining %them right to bypass.

{\begin{remark}
\label{zero_pole}
Note that in both representations (\ref{Mellin-Barnes1}) and (\ref{Mellin-BarnesT2_1}) we cannot include the term corresponding to the pole at $s = 0$ into the integral term, since in this case either ${\mathcal L}_{+\infty}$ or ${\mathcal L}_{-\infty}$ should cross the branch cut of the corresponding multi-valued function.
\end{remark}
}

\section{Laplace transforms of the Le Roy type-function}\label{S:Laplace}

\setcounter{section}{3}
\setcounter{equation}{0}\setcounter{theorem}{0}

Let us consider the case $\gamma > 1$ and calculate the Laplace transform pair related to the Le Roy type-function by means of an expression which is similar to that used to obtain the Laplace transform of the Mittag-Leffler function

%Let us consider the case $\gamma > 1$ and calculate the Laplace transform of the following expression which is similar to that used in the Laplace transforms of the Mittag-Leffler function:

{\begin{lemma}
\label{Laplace}
Let $\alpha, \beta > 0$, $\gamma > 1$ be positive numbers, $\lambda\in {\mathbb C}$. The Laplace transform of the Le Roy type-function is
\begin{equation}
\label{Laplace1}
{\mathcal L} \left\{ t^{\beta - 1} {F}_{\alpha,\beta}^{(\gamma)}(\lambda t^{\alpha})\right\}(s) = \frac{1}{s^{\beta}} {F}_{\alpha,\beta}^{(\gamma - 1)}(\lambda s^{-\alpha}).
\end{equation}
\end{lemma}
}
\begin{proof}
For the above mentioned values of its parameters ${F}_{\alpha,\beta}^{(\gamma)}(\cdot)$ is an entire function of its argument. Therefore the below interchanging of the integral and the sum holds
\[
	\begin{aligned}
	{\mathcal L} \left\{ t^{\beta - 1} {F}_{\alpha,\beta}^{(\gamma)}(\lambda t^{\alpha})\right\}(s)
	& = \int\limits_{0}^{\infty} \eu^{- s t} \sum\limits_{k=0}^{\infty} t^{\beta - 1} \frac{\lambda^k t^{\alpha k}}{[\Gamma(\alpha k + \beta)]^{\gamma}}  \du t = \\
	&= \sum\limits_{k=0}^{\infty} \frac{\lambda^k }{[\Gamma(\alpha k + \beta)]^{\gamma}} \int\limits_{0}^{\infty} \eu^{- s t} t^{\beta - 1} t^{\alpha k}  \du t = \\
	&= \sum\limits_{k=0}^{\infty} \frac{\lambda^k}{[\Gamma(\alpha k + \beta)]^{\gamma}} \frac{\Gamma(\alpha k + \beta)}{s^{\alpha k + \beta}} =
\frac{1}{s^{\beta}} {F}_{\alpha,\beta}^{(\gamma - 1)}(\lambda s^{-\alpha}).
\end{aligned}
\]
which allows to conclude the proof.
\end{proof}

\begin{corollary}\label{Laplace2}
For particular values of the parameter $\gamma$ formula (\ref{Laplace1}) allows to establish the following simple relationships between the Laplace transform of the Le Roy type-function and the Mittag-Leffler function:

\begin{eqnarray}
	\gamma = 2 &:& \quad
{\mathcal L} \left\{ t^{\beta - 1} {F}_{\alpha,\beta}^{(2)}(\lambda t^{\alpha}) \right\} (s) = \frac{1}{s^{\beta}} {E}_{\alpha,\beta}(\lambda s^{-\alpha}), \label{Laplace3} \\
	\gamma = 3 &:& \quad
{\mathcal L} \left\{ t^{\beta - 1} {F}_{\alpha,\beta}^{(3)}(\lambda t^{\alpha})\right\} (s) = \frac{1}{s^{\beta}} {E}_{\alpha,\beta; \alpha,\beta}(\lambda s^{-\alpha}), \label{Laplace4} \
\end{eqnarray}
where ${E}_{\alpha,\beta}$ and  ${E}_{\alpha,\beta; \alpha,\beta}$ are respectively the 2-parameter and 4-parameter Mittag-Leffler functions in the sense of (\ref{LKK}) (see \cite[Chs. 4, 6]{GKMR}).
\end{corollary}

For any arbitrary positive integer value of the parameter $\gamma$ the Laplace transform of the Le Roy type-function can be represented in terms of the generalized Wright function (see \cite[Appendix F]{GKMR})
\begin{equation}\label{eq:Wright}
	_{p}\Psi_{q}(z) \equiv \,
	_{p}\Psi_{q}(z)\left[ \begin{array}{c} (\rho_1,a_1),\dots,(\rho_p,a_p) \\ (\sigma_1,b_1),\dots,(\sigma_q,b_q) \end{array} ; z \right] = \sum_{k=0}^{\infty} \frac{z^{k}}{k!} \frac
		{\displaystyle\prod_{r=1}^{p} \Gamma(\rho_{r}k + a_{r}) }
		{\displaystyle\prod_{r=1}^{q} \Gamma(\sigma_{r} k + b_{r}) }
	,
\end{equation}
where $p$ and $q$ are integers and $\rho_r, a_r, \sigma_r,b_r$ are real or complex parameters.

{\begin{lemma}
\label{Laplace5}
Let $\alpha, \beta > 0$, $\gamma = m\in {\mathbb N}$ be positive numbers. The Laplace transforms of the Le Roy type-function is given by
\begin{equation}\label{Laplace6}
{\mathcal L} \left\{ {F}_{\alpha,\beta}^{(m)}(t)\right\}(s) = \frac{1}{s}\;  {}_{2}\Psi_{m}\left(\left[\begin{array}{c} (1, 1), (1, 1) \\
\underbrace{(\beta, \alpha), \ldots, (\beta, \alpha)}\limits_{m-\text{times}}\end{array}\right]  ; \frac{1}{s}\right).
\end{equation}
\end{lemma}
}
Formula (\ref{Laplace6}) is obtained %in a straightforward way
 directly by using definitions of the Laplace transforms and the generalized Wright function
(cf. \cite[p. 44]{KiSrTr06}).

\section{Asymptotic on the negative semi-axis}\label{S:Asympt}

\setcounter{section}{4}
\setcounter{equation}{0}\setcounter{theorem}{0}

In this section we study the asymptotic expansion of the Le Roy type-function for large arguments. In particular, we pay attention to the case of a positive integer parameter $\gamma = m\in {\mathbb N}$ and, with major emphasis, we discuss the behavior of the function along the negative real semi-axis.

Since in this case (integer positive $\gamma = m$), the Le Roy type-function is a particular instance of the generalized Wright function (\ref{eq:Wright}), namely
\[
	F_{\alpha,\beta}^{(m)}(z) = \, _{1}\Psi_{m}(z),
\]
with $\rho_1=1$, $a_1=1$, $\sigma_1=\sigma_2=\dots=\sigma_m=\alpha$ and $b_1=b_2=\dots=b_m=\beta$, some of the results on the expansion of the Wright function, discussed first in \cite{Wright1940,Wright1940_PTRSL} and successively in \cite{Braaksma1964,Paris2010}, can be exploited to derive suitable expansions of the Le Roy type-function.

In particular, by applying to $F_{\alpha,\beta}^{(m)}(z)$ the reasoning proposed in \cite{Paris2010}, we introduce the functions
\begin{equation}\label{eq:H_fun}
	H(z) = \sum_{k=0}^{\infty} \frac{ (-1)^k z^{-(k+1)}}{[\Gamma(\beta-\alpha (k+1))]^m}
	= - \sum_{k=1}^{\infty} \frac{ (-1)^k z^{-k}}{[\Gamma(\beta-\alpha k)]^m}
\end{equation}
and
\[
	E(z) = m^{\frac{1}{2}(m+1)-m\beta} z^{\frac{m+1-2 m\beta}{2\alpha m}} \eu^{m z^{\frac{1}{\alpha m}}} \sum_{j=0}^{\infty} A_j m^{-j} z^{-\frac{j}{\alpha m}} ,
\]
where $A_j$ are the coefficients in the inverse factorial expansion of
\[
	\frac{\Gamma(\alpha m s + \theta')}{\bigl[\Gamma(\alpha s+\beta)\bigr]^m} = \alpha m  \sum_{j=0}^{M-1} \frac{A_j}{(\alpha m s + \theta')_j} + \frac{{\mathcal O}(1)}{(\alpha m s + \theta')_M}
		, 
	\theta' = m \beta - \frac{m-1}{2},
\]
with $(x)_j = x (x+1)\cdots(x+j-1)$ denoting the Pochhammer symbol. The following results directly descend from Theorem 1, 2 and 3 in \cite{Paris2010}.

\begin{theorem}\label{thm:Expans1}
Let $m \in {\mathbb N}$ and $0<\alpha m < 2$. Then
\[
	F_{\alpha,\beta}^{(m)}(z) \sim
	\left\{\begin{array}{ll}
	E(z) + H(z \eu^{\mp \pi \iu}), & \textrm{if } |\arg z| \le \frac{1}{2}\pi \alpha m, \\
\\
	H(z \eu^{\mp \pi \iu}), & \textrm{otherwise }, \\
	\end{array}\right.
	 \quad \textrm{as } |z| \to \infty,
\]
with the upper or lower signs chosen according as $\arg z>0$ or $\arg z <0$ respectively.
\end{theorem}

\begin{theorem}\label{thm:Expans2}
Let $m \in {\mathbb N}$, $\alpha m = 2$ and $|\arg z | \le \pi$. Then
\[
	F_{\alpha,\beta}^{(m)}(z) \sim E(z) + E(z \eu^{\mp 2\pi \iu}) + H(z \eu^{\mp \pi \iu}) , \quad \textrm{as } |z| \to \infty,
\]
with the upper or lower signs chosen according as $\arg z>0$ or $\arg z <0$ respectively.
\end{theorem}

\begin{theorem}\label{thm:Expans3}
Let $m \in {\mathbb N}$, $\alpha m > 2$ and $|\arg z | \le \pi$. Then
\[
	F_{\alpha,\beta}^{(m)}(z) \sim \sum_{r=-P}^{P} E(z \eu^{2\pi \iu r}), \quad \textrm{as } |z| \to \infty,
\]
with $P$ the integer number such that $2P+1$ is the smallest odd integer satisfying $2P+1>\frac{1}{2}m \alpha$.
\end{theorem}

Deriving the coefficients $A_j$ in $E(z)$ is a quite cumbersome process (a sophisticated algorithm is however described in \cite{Paris2010}). Anyway the first coefficient
\[
	A_0 = \frac{1}{\alpha}(2\pi)^{(1-m)/2} m^{-1 -\frac{1}{2}m + m\beta}
\]
is explicitly available, thus allowing to write
\begin{equation}\label{eq:E_fun}
	E(z) = a_{0} z^{\frac{m+1-2m\beta}{2\alpha m}} \eu^{m z^{\frac{1}{\alpha m}}} \left( 1 + {\mathcal O} ( z^{-\frac{1}{\alpha m}} ) \right)
	,
\end{equation}
where
\[
a_{0} = \frac{1}{\alpha\sqrt{m}} (2\pi)^{(1-m)/2} .
\]

We are then able to represent the asymptotic behavior of the Le Roy type-function on the real negative semi axis by means of the following theorem.

\begin{theorem}\label{thm:F_asy}
Let $\alpha >0$, $m \in {\mathbb N}$ and $t > 0$. Then
\[
	F_{\alpha,\beta}^{(m)}(-t) \sim
	\left\{\begin{array}{ll}
		H(t), & 0 < \alpha m < 2, \\
		G(t) + H(t), & \alpha m = 2, \\
		G(t), & 2 < \alpha m,  \\
	\end{array}\right.
	, \quad
	t \to \infty
\]
where $H(t)$ is the same function introduce in (\ref{eq:H_fun}) and
\[
	G(t) = 2 a_0 t^{\frac{m+1-2m\beta}{2\alpha m}} \exp\Bigl(m t^{\frac{1}{\alpha m}} \cos {\textstyle \frac{\pi}{\alpha m}}\Bigr)  \cos \Bigl( \textstyle\frac{\pi(m+1-2m\beta)}{2\alpha m}  + m t^{\frac{1}{\alpha m}} \sin {\textstyle \frac{\pi}{\alpha m}} \Bigr) .
\]
\end{theorem}

\begin{proof}
Since for real and negative values $z=-t$, with $t>0$, we can write $z=\eu^{\iu \pi}t$, the use of Theorems \ref{thm:Expans1}, \ref{thm:Expans2}  and \ref{thm:Expans3} allows to to describe the asymptotic behavior of the Le Roy type-function along the negative semi-axis according to
\[
	F_{\alpha,\beta}^{(m)}(-t) \sim
	\left\{\begin{array}{ll}
		H(t), & 0 < \alpha m < 2, \\
		E(\eu^{\iu \pi } t) + E(\eu^{- \iu \pi } t) + H(t), & \alpha m = 2, \\
		%E(\eu^{-\pi \iu} t) + E(\eu^{\pi \iu} t) + E(\eu^{3\pi \iu} t) & 2 < \alpha m < 6 \\
		%E(\eu^{-3\pi \iu} t) + E(\eu^{-\pi \iu} t) + E(\eu^{\pi \iu} t) + E(\eu^{3\pi \iu} t) + E(\eu^{5 \pi \iu} t) & 6 \le \alpha m < 10 \\
	\end{array}\right.
	, \quad t \to \infty
\]
and when $\alpha m > 2$ it is for an integer $P\ge 1$
\begin{equation}\label{eq:F_asy_Large_alpham}
	F_{\alpha,\beta}^{(m)}(-t) \sim \sum_{r=-P}^{P} E(\eu^{ \iu (2r+1) \pi} t), \quad 2(2P-1) \le \alpha m < 2(2P+1) .
\end{equation}

We denote, for shortness,
\[
        \phi_r(t) = \frac{r \pi(m+1-2m\beta)}{2\alpha m}  + m t^{\frac{1}{\alpha m}} \sin \frac{r\pi}{\alpha m} ,
\]
and, by means of some standard trigonometric identities, we observe that
\[
        E(\eu^{\iu r\pi } t) = a_0 t^{\frac{m+1-2m\beta}{2\alpha m}} \exp\Bigl(m t^{\frac{1}{\alpha m}} \cos\textstyle\frac{r \pi}{\alpha m}\Bigr)
        \biggl[ \cos \phi_r(t) + \iu \sin \phi_r(t) \biggr],
\]
from which it is immediate to see that
\[
        E(\eu^{\iu r \pi } t) + E(\eu^{- \iu r \pi } t)  = 2 a_0 t^{\frac{m+1-2m\beta}{2\alpha m}} \exp\Bigl(m t^{\frac{1}{\alpha m}} \cos {\textstyle \frac{r\pi}{\alpha m}}\Bigr)  \cos \phi_r(t) ,
\]
and, clearly, for $\alpha m < 2(2P+1)$ it is
\[
        \lim_{t\to \infty} E(\eu^{\iu (2P+1) \pi } t) = 0.
\]

Therefore, after introducing the functions
\[
	G_r(t) = 2 a_0 t^{\frac{m+1-2m\beta}{2\alpha m}} \exp\Bigl(m t^{\frac{1}{\alpha m}} \cos {\textstyle \frac{r\pi}{\alpha m}}\Bigr)  \cos \Bigl( \textstyle\frac{r\pi(m+1-2m\beta)}{2\alpha m}  + m t^{\frac{1}{\alpha m}} \sin {\textstyle \frac{r\pi}{\alpha m}} \Bigr)
\]
for $r=1,2,\dots,P$, with
\[
P = \left\lfloor \frac{1}{2} \left(\frac{\alpha m}{2} +1 \right) \right\rfloor
\]
and $\left\lfloor x \right\rfloor$ the greatest integer smaller than $x$, we can summarize the asymptotic behavior of the Le Roy type-function as $t \to \infty$ by means of
\begin{equation}\label{eq:F_asy_summarized}
	F_{\alpha,\beta}^{(m)}(-t) \sim
	\left\{\begin{array}{ll}
		H(t), & 0 < \alpha m < 2, \\
		G_1(t) + H(t), & \alpha m = 2, \\
		\displaystyle\sum_{r=1}^{P} G_r(t),  \quad    & \alpha m > 2 . \\
	\end{array}\right.
\end{equation}

Observe now that since $\cos {\textstyle \frac{\pi}{\alpha m}} > \cos {\textstyle \frac{2\pi}{\alpha m}} > \dots > \cos {\textstyle \frac{P\pi}{\alpha m}} > 0$ the exponential in $G_1(t)$ dominates the exponential in the others $G_r(t)$, $r\ge 2$, which can be therefore neglected for $t\to\infty$ and hence the proof follows after putting $G(t)=G_1(t)$.
\end{proof}

We note that the asymptotic representation for $\alpha m < 2$ is similar to a well-known representation for the 2-parameter Mittag-Leffler function used in \cite{GoLoLu02} also for computational purposes.

As we can clearly observe, $\alpha m = 2$ is a threshold value (compare with (\ref{OrderTypeGGP}) for the order of this entire function) for the asymptotic behavior of $F_{\alpha,\beta}^{(m)}(-t)$ as $t\to \infty$. Whenever $\alpha m <2$ the function is expected to decay in an algebraic way, while for $\alpha m >2$ an increasing but oscillating behavior is instead expected.

We now presents some plots of $F_{\alpha,\beta}^{(m)}(-t)$ and we make a comparison with the asymptotic expansions obtained by Theorem \ref{thm:F_asy}.

We must observe that the numerical evaluation of $F_{\alpha,\beta}^{(m)}(-t)$ has not been so far investigated and, surely, this topic deserves some attention which is however beyond the scope of this paper. To obtain reference values to be compared with the asymptotic expansion, we have therefore directly evaluated a large number of the first terms of the series (\ref{GGP1}) until numerical convergence, namely until there are achieved terms so small (under the precision machine) to be neglected. To avoid that numerical cancelation and round-off errors affect in a remarkable way the results, we have used the high precision arithmetic of Maple 15 and evaluated $F_{\alpha,\beta}^{(m)}(-t)$ with 2000 digits (standard computation is just 16 digits).

Only the first terms of $H(t)$ in the asymptotic expansion are used in the plots; namely, the function $H(t)$ is replaced by
\[
	H_{K}(t) =  - \sum_{k=1}^{K} \frac{ (-1)^k z^{-k}}{[\Gamma(\beta-\alpha k)]^m},
\]
which turns out to be accurate enough for large or moderate values of $t$ and $K$ (the selected value of $K$ is indicated in each plot). In most cases the plots of the Le Roy type-function and those of its expansion are almost identical, thus confirming the theoretical findings.

When $\alpha m < 2$ since the presence of just negative powers of $t$ we clearly expect, over long intervals of $t$, a decreasing behavior as shown in Figure \ref{fig:Exp_alpham18}.

\begin{figure}[ht]
\centering
\includegraphics[width=.70\textwidth]{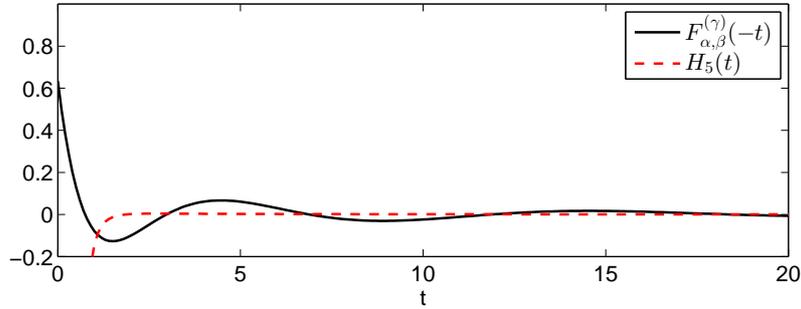}
\caption{Comparison of $F_{\alpha,\beta}^{(m)}(-t)$  with its asymptotic expansion for $\alpha=0.6$, $\beta=0.8$ and $\gamma=m=3$ (here $\alpha m<2$).}
\label{fig:Exp_alpham18}
\end{figure}

For the special case $\alpha m=2$  we show the behavior of the Le Roy type-function when $m+1-2m\beta<0$ (Figure \ref{fig:Exp_alpham2_beta1}) and the one obtained when $m+1-2m\beta >0$ (Figure \ref{fig:Exp_alpham2_beta2}). In both cases the exponential term in $G(t)$ becomes a constant and the leading term is of algebraic type, respectively with a negative and a positive power, thus justifying the decreasing and the increasing amplitude of the oscillations related to the presence of the cosine function.

\begin{figure}[ht]
\centering
\includegraphics[width=.70\textwidth]{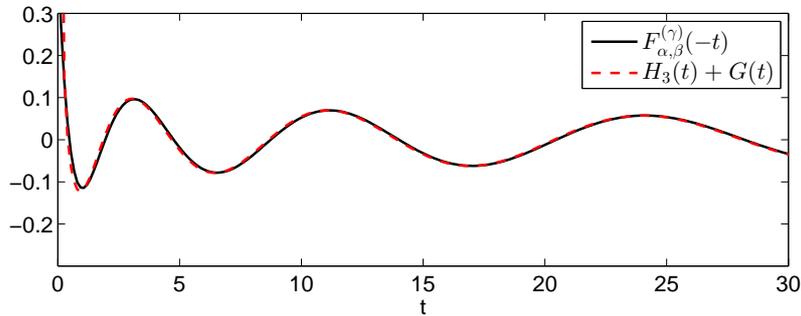}
\caption{Comparison of $F_{\alpha,\beta}^{(m)}(-t)$  with its asymptotic expansion for $\alpha=0.5$, $\beta=0.75$ and $\gamma=m=4$ (here $\alpha m=2$ and $m+1-2m\beta <0$).}
\label{fig:Exp_alpham2_beta1}
\end{figure}

\begin{figure}[ht]
\centering
\includegraphics[width=.70\textwidth]{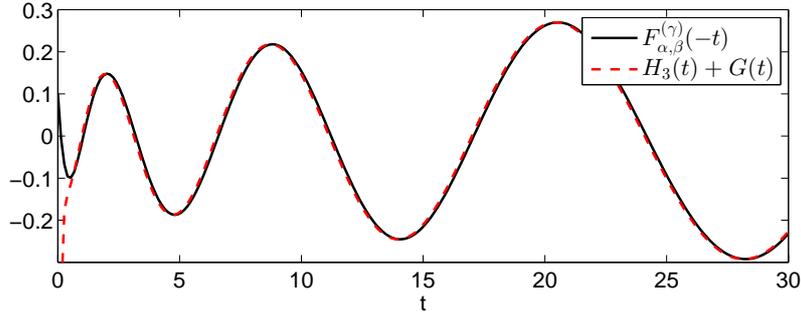}
\caption{Comparison of $F_{\alpha,\beta}^{(m)}(-t)$  with its asymptotic expansion for $\alpha=0.5$, $\beta=0.5$ and $\gamma=m=4$ (here $\alpha m=2$ and $m+1-2m\beta >0$).}
\label{fig:Exp_alpham2_beta2}
\end{figure}

Whenever $\alpha m >2$ an oscillating behavior is always expected due to the presence of the cosine function in $G(t)$. When $\beta$ is selected such that $m+1-2m\beta <0$ the amplitude of the oscillations could decrease within an interval of the argument $t$ of moderate size as a consequence of the algebraic term with a negative power, as shown in Figure \ref{fig:Exp_alpham21a}.

\begin{figure}[ht]
\centering
\includegraphics[width=.70\textwidth]{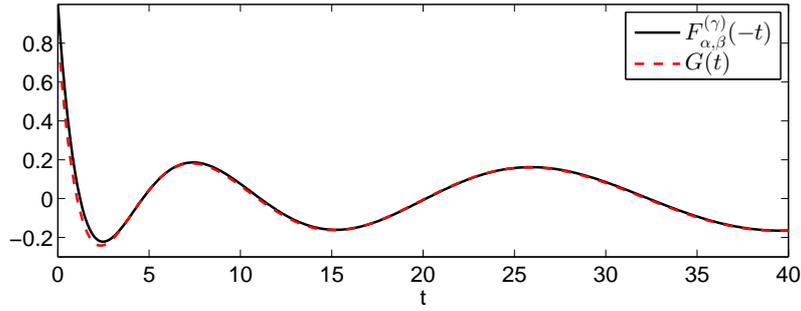}
\caption{Comparison of $F_{\alpha,\beta}^{(m)}(-t)$  with its asymptotic expansion for $\alpha=0.7$, $\beta=1.0$ and $\gamma=m=3$ (here $\alpha m=2.1$ and $m+1-2m\beta <0$).}
\label{fig:Exp_alpham21a}
\end{figure}

However the positive argument of the exponential will necessary lead to a growth for larger values of $t$ as clearly shown in Figure \ref{fig:Exp_alpham21b} where the same function of Figure \ref{fig:Exp_alpham21a} is plotted but on a wider interval of the argument $t$.

\begin{figure}[ht]
\centering
\includegraphics[width=.70\textwidth]{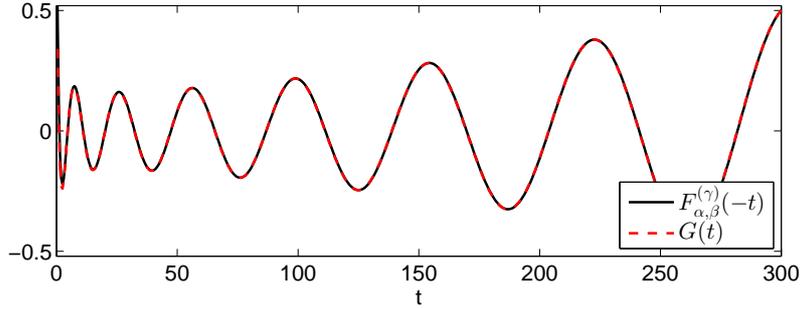}
\caption{Long range behavior of $F_{\alpha,\beta}^{(m)}(-t)$ for $\alpha=0.7$, $\beta=1.0$ and $\gamma=m=3$ (here $\alpha m=2.1$ and $m+1-2m\beta <0$).}
\label{fig:Exp_alpham21b}
\end{figure}

When $\beta$ is selected such that $m+1-2m\beta >0$ no initial decay is instead expected.

\begin{remark}\label{AsyRepr_negative3} If  $- \alpha j + \beta$ is integer for some $j\in {\mathbb N}$, the corresponding term in (\ref{eq:H_fun}) disappears. In particular, it follows from (\ref{eq:F_asy_summarized}) that the Le Roy type-function has the algebraic decay ${F}_{\alpha,\beta}^{(m)}(-t) = O\left(t^{-1}\right)$ as $t \rightarrow \infty$ if $\alpha\not= \beta$ and ${F}_{\alpha,\beta}^{(m)}(-t) = O\left(t^{-2}\right)$ if $\alpha = \beta$.
\end{remark}

\begin{remark}
The different behavior for $\alpha m < 2$ and $\alpha m > 2$ is consistent with the estimates presented (without a proof) in the Olver's book \cite[Ex. 84, page 309]{Olv74} for the special case $R_{\rho}(-t^{\rho})$ of the Le Roy function (\ref{LeR}), namely
\[
R_{\rho}(-t^{\rho}) \sim \left\{ \begin{array}{ll}
\frac{\rho^{-\rho}}{\Gamma(1 - \rho)} ( t \log t)^{-\rho} \left[1 - \frac{\rho c_0}{\log t} + O\left(\frac{1}{\log^2 t}\right)\right], & 1 < \rho < 2, \\
& \\
\tilde{a}_0 t^{(1-\rho)/2} \eu^{\rho t \cos\frac{\pi}{\rho}} \left[\sin \left(\frac{\pi}{\rho} + t \rho \sin \frac{\pi}{\rho}\right) + O\left(\frac{1}{t^2}\right)\right], & \rho > 2,
\end{array}
\right.
\]
with $\tilde{a}_0 =\frac{2}{\sqrt{\rho}} (2\pi)^{(1-\rho)/2}$ and $c_0$ is the Euler constant.
\end{remark}

\begin{remark}
For large $\alpha m \gg 2$ the expansion provided by Theorem \ref{thm:F_asy} could be not very accurate since slow decaying terms are neglected in $E(z)$. In this case it would be advisable to evaluate further coefficients $A_j$ by means, for instance, of the algorithm described in \cite{Paris2010}.
\end{remark}

\begin{remark}\label{num3}
In the case of non-integer values of $\gamma$ the Le Roy type-function cannot be expressed in terms of the Wright function and therefore Theorem \ref{thm:F_asy} no longer applies; some different techniques need to be developed to adequately treat the case of non-integer $\gamma$.
\end{remark}

\section{Extension to negative values of the parameter $\alpha$}\label{S:Negative}

\setcounter{section}{5}
\setcounter{equation}{0}\setcounter{theorem}{0}

The ${\mathcal L}_{-\infty}$-integral representation can be used to extend the function $F_{\alpha,\beta}^{(\gamma)}(z)$ to negative values of the parameter $\alpha$ (we follow here the approach described in \cite{KiKoRo13}). To clearly distinguish the two cases we denote this extended Le Roy type-function as ${\mathcal F}_{-\alpha,\beta}^{(\gamma)}(z)$.

\begin{definition}
\label{extension}
The function ${\mathcal F}_{-\alpha,\beta}^{(\gamma)}(z)$, $\alpha, \beta, \gamma$ is defined by the following relation
\begin{equation}
\label{extension1}
{\mathcal F}_{-\alpha,\beta}^{(\gamma)}(z) = -
\frac{1}{2\pi \iu} \int\limits_{{\mathcal L}_{-\infty}} \frac{\Gamma(- s) \Gamma(1 + s)}{[\Gamma(-\alpha s + \beta)]^{\gamma}} (- z)^{s} \du s,
\end{equation}
where ${\mathcal L}_{-\infty}$ is a right loop situated in a horizontal strip starting at the point $- \infty + i \varphi_1$ and terminating at the point $- \infty + i \varphi_2$, $-\infty < \varphi_1 < 0< \varphi_2 < + \infty$, crossing the real line at a point $c, -1 < c < 0$,  values of $(- z)^s$  are calculated as described above, and the branch of the multi-valued function $[\Gamma(-\alpha s + \beta)]^{\gamma}$ is defined in the complex plane cut along the positive semi-axes starting from $\frac{\beta}{\alpha}$, ending at $+\infty$, with $[\Gamma(-\alpha x + \beta)]^{\gamma}$ being positive for all negative $x$.
\end{definition}

Using the slight correction of the proof of Theorem \ref{Mellin-Barnes} we get the following result.
\begin{theorem}
\label{Mellin-Barnes-ext}
Let $\alpha, \beta, \gamma > 0$, then the extended Le Roy type-function (\ref{extension1}) satisfies the following series representation
\begin{equation}
\label{Mellin-Barnes-ext1}
{\mathcal F}_{-\alpha,\beta}^{(\gamma)}(z) = - \sum\limits_{k=1}^{\infty} \frac{1}{[\Gamma(\alpha k + \beta)]^{\gamma}} \frac{1}{z^k},\;\;\; z\in {\mathbb C}\setminus\{0\}.
\end{equation}
\end{theorem}

\begin{corollary}
\label{Mellin-Barnes-ext2}
The Le Roy type-function (\ref{GGP1}) and its extension (\ref{extension1}) are connected via the following relation
\begin{equation}
\label{Mellin-Barnes-ext3}
{\mathcal F}_{-\alpha,\beta}^{(\gamma)}(z) = \frac{1}{[\Gamma(\beta)]^{\gamma}} - {F}_{\alpha,\beta}^{(\gamma)}\left(\frac{1}{z}\right).
\end{equation}
\end{corollary}

Observe that the relation (\ref{Mellin-Barnes-ext3}) is similar to the ones presented in \cite{Han09,KiKoRo13}.

\section{Outline and discussion}\label{S:Disc}

\setcounter{section}{6}
\setcounter{equation}{0}\setcounter{theorem}{0}

From the proof of Theorem \ref{Mellin-Barnes} it follows that using integrals (\ref{Mellin-Barnes1}) and (\ref{extension1}) one can define the corresponding functions for the following values of parameters ${\mathrm{Re}}\, \alpha > 0, \beta\in {\mathbb C}, \gamma > 0$
in the case  of the function ${F}_{\alpha,\beta}^{(\gamma)}\left({z}\right)$, and ${\mathrm{Re}}\, \alpha < 0, \beta\in {\mathbb C}, \gamma > 0$ in the case  of the function
 ${\mathcal F}_{-\alpha,\beta}^{(\gamma)}(z)$.

In our Section \ref{S:Asympt} we perform a brief asymptotic analysis of the behavior of the Le Roy type-function on the negative semi-axes. It is given only for integer positive parameter $\gamma$. Further study of such a behavior is of special importance due to some applications involving this function. The Laplace method will be used to attack this problem, it is a subject of the forthcoming paper.

In order to develop fractional type probability distributions aiming various models in different branches of science it is important, in particular, to find the cases of complete monotonicity (see, e.g., \cite{MilSam}) of the involved special functions. It is true also for Le Roy type-function since its relationship to the Convey-Maxwell-Poisson distribution.

Below we formulate a conjecture and an open question which are important for the Le Roy type-function.

{\bf Conjecture.}
{\it For ``small'' values of parameters, namely $0 < \gamma (\beta - 1/2) < 1$ the Le Roy type-function has the following asymptotics on the negative semi-axes}
$$
{F}_{\alpha,\beta}^{(\gamma)}\left(- x\right) = c \frac{x^{\frac{1/2 - \gamma (\beta - 1/2)}{\alpha \gamma}}}{\log x},\;\;\; x \to +\infty,
$$
{\it where $c$ is an absolute constant.}

{\bf Open question}. {\it  To find conditions on the parameters $\alpha,\beta,\gamma$ for which the function $F_{\alpha,\beta}^{(\gamma)}(-x), 0 < x < +\infty,$ is completely monotone (cf., \cite{MilSam}).}

\section*{Acknowledgements}
 The work of R. Garrappa has been supported by the INdAM GNCS Project 2017. The work of F. Mainardi has been carried out in the framework of the INdAM GNFN activity. The work of S. Rogosin is partially supported by the People Programme (Marie Curie Actions) of the European Union Seventh Framework Programme FP7/2007-
2013/ under REA grant agreement PIRSES-GA-2013-610547 - TAMER, by ISA (Institute for Advanced Studies) Bologna University and by Belarusian Fund for Fundamental Scientific Research (grant F17MS-002).

The authors are grateful to Professors Virginia Kiryakova and Yuri Luchko for their valuable comments which improved the presentation of the results of the paper.

  %%%%%%%%%%%%%%%%%%%%%%%%%%%%%%%%

\end{document}